\documentclass{amsart}
\usepackage{amsmath}
\usepackage{amsfonts}

\newtheorem{theorem}{Theorem}
\theoremstyle{plain}

\newtheorem{corollary}{Corollary}

\newtheorem{remark}{Remark}

\numberwithin{equation}{section}

\begin{document}
\title[Weighted Ostrowski Type Inequality ]{A Weighted Ostrowski Type
inequality for L$_{1}\left[ a,b\right] $ and applications}
\author{A. Qayyum$^{1,2}$}
\address{$^{1}$Department of Mathematics, Universiti teknologi Patronas,
Malaysia. $^{2}$Department of Mathematics, University of Hail, Hail 2440,
Saudi Arabia}
\email{atherqayyum@gmail.com}
\author{S. S. Dragomir$^{1,2}$}
\address{$^{1}$Mathematics, College of Engineering \& Science, Victoria
University, PO Box 14428, Melbourne City, MC 8001, Australia.\\
$^{2}$School of Computational and Applied Mathematics, University of the
Witwatersrand, Private Bag 3, Johannesburg 2050, South Africa.}
\email{sever.dragomir@vu.edu.au}
\author{M. Shoaib}
\address{ Department of Mathematics, University of Hail, Hail 2440, Saudi
Arabia}
\email{safridi@gmail.com}
\date{Today}
\thanks{This paper is in final form and no version of it will be submitted
for publication elsewhere.}

\begin{abstract}
The aim of this paper is to obtain some generalized weighted Ostrowski
inequalities for differentiable mappings. Some well known inequalities can
be derived as special cases of the inequalities obtained here. In addition,
perturbed mid-point inequality and perturbed trapezoid inequality are also
obtained. The inequalities obtained here have direct applications in
Numerical Integration, Probability Theory, Information Theory and Integral
Operator Theory. Some of these applications are discussed.
\end{abstract}

\maketitle

\section{Introduction}

Inequalities appear in most of the domains of Mathematics and has
applications in numerical integration, probability theory, information
theory and integral operator theory. Inequalities as a field came into
promince with the publication of a book by Hardy, Littlewood and Polya \cite%
{6} in 1934. In 1938, Ostrowski \cite{8} discovered a useful inequality,
which is known after his name as Ostrowski inequality. In many practical
investigation, it is necessary to bound one quantity by another. This
classical Ostrowski inequality is very useful for this purpose\textbf{.}
Beckenbach and Bellman \cite{2} \ and Mitrinovi\'{c} \cite{12} highlighted
the importance of inequalities in their respective publications.\bigskip

More recently new inequalities of Ostrowski type were presented Dragomir and
Wang \textbf{\cite{5}} in 1997 and Dragomir and Rassias \cite{4} in 2002.
The weighted version of Ostrowski inequality was first presented in 1983 by
Pecari\'{c} and Savi\'{c} \cite{9}. In 2003, Roumeliotis \cite{11} did some
improvement in the weighted version of Ostrowski-Gr\"{u}ss type
inequalities. In \cite{10} and \cite{7}, Qayyum and Hussain discussed the
weighted version of Ostrowski-Gr\"{u}ss type inequalities. The tools that
are used in this paper are weighted Peano kernel approach which is the
classical and extensively used approach in developing Ostrowski integral
inequalities. The results presented in this paper are very general in
nature. \ The inequalities proved by Dragomir et al \textbf{\cite{5}},
Barnett et al \textbf{\cite{1}} and Cerone et al \textbf{\cite{3}} are
special cases of the inequalities developed here.

Ostrowski \cite{8} proved the classical integral inequality which is stated
here without proof.

\begin{theorem}
Let \ \ $f\ $: $\left[ a,b\right] \rightarrow
\mathbb{R}
$ \ be\ continuous on $\left[ a,b\right] $ and differentiable on $\left(
a,b\right) ,$ whose derivative $f^{\prime }:\left( a,b\right) \rightarrow
\mathbb{R}
$ is bounded on $\ \left( a,b\right) ,$ i.e. $\left\Vert f^{\prime
}\right\Vert _{\infty }=\sup_{t\in \left[ a,b\right] }\left\vert f^{\prime
}\left( t\right) \right\vert <\infty $ then%
\begin{equation}
\left\vert f(x)-\frac{1}{b-a}\int_{a}^{b}f(t)dt\right\vert \leq \left[ \frac{%
1}{4}+\frac{\left( x-\frac{a+b}{2}\right) ^{2}}{\left( b-a\right) ^{2}}%
\right] \left( b-a\right) \left\Vert f^{\prime }\right\Vert _{\infty }
\tag{1.1}
\end{equation}%
for all $x\in \left[ a,b\right] $. The\ constant$\ \frac{1}{4}\;$is sharp in
the sense that it can not be replaced by a smaller one.
\end{theorem}

Dragomir and Wang \cite{5} proved $\left( 1.1\right) $ for $f^{\text{ }%
\prime }\in L_{1}\left[ a,b\right] $ $,$ as follows:

\begin{theorem}
Let \ \ $f\ $:$I\subseteq $ $%
\mathbb{R}
\rightarrow
\mathbb{R}
$ \ be\ a differentiable mapping in $I^{\circ }$ and $a,b\in I^{\circ }$with
$a<b.$ If $f^{\text{ }\prime }\in L_{1}\left[ a,b\right] $, then the
inequality holds%
\begin{equation}
\left\vert \text{ }f(x)-\frac{1}{b-a}\int_{a}^{b}f(t)dt\right\vert \leq %
\left[ \frac{1}{2}+\frac{\left\vert x-\frac{a+b}{2}\right\vert }{b-a}\right]
\left\Vert f^{\prime }\right\Vert _{1}  \tag{1.2}
\end{equation}%
for all $x\in \left[ a,b\right] $.
\end{theorem}

They also pointed out some applications of $(1.2)$ in Numerical Integration
as well as for special means.

Barnett et,al, \cite{1} proved out an inequality of Ostrowski type for twice
differentiable mappings which is in terms of the $\left\Vert .\right\Vert
_{1}$ norm of the second derivative $f^{\prime \prime }$ and apply it in
numerical integration and for some special means.

The following inequality of Ostrowski's type for mappings which are twice
differentiable, holds \cite{3}.

\begin{theorem}
Let \ \ $f\ $: $\left[ a,b\right] \rightarrow
\mathbb{R}
$ \ be\ continuous on $\left[ a,b\right] $ and twice differentiable in $%
\left( a,b\right) $ and $f^{\text{ }\prime \prime }\in L_{1}\left(
a,b\right) .$ Then the inequality obtained%
\begin{eqnarray}
&&\left\vert \text{ }f(x)-\frac{1}{b-a}\int_{a}^{b}f(t)dt-\left( x-\frac{a+b%
}{2}\right) f^{\text{ }\prime }\left( x\right) \right\vert  \label{1.3} \\
&\leq &\frac{1}{2\left( b-a\right) }\left( \left\vert x-\frac{a+b}{2}%
\right\vert +\frac{1}{2}\left( b-a\right) \right) ^{2}\left\Vert f^{\prime
\prime }\right\Vert _{1}  \notag
\end{eqnarray}%
for all $x\in \left[ a,b\right] $.
\end{theorem}

J. Roumeliotis \cite{4}, presented product inequalities and weighted
quadrature. The weighted inequlity was also obtained in Lebesgue spaces
involving first derivative of the function, which is given by%
\begin{eqnarray}
&&\left\vert \text{ }\frac{1}{b-a}\int_{a}^{b}w\left( t\right)
f(t)dt-m\left( a,b\right) f\left( x\right) \right\vert  \notag \\
&\leq &\frac{1}{2}\left[ m\left( a,b\right) +\left\vert m\left( a,x\right)
-m\left( x,b\right) \right\vert \right] \left\Vert f^{\prime \prime
}\right\Vert _{1}  \label{1.4}
\end{eqnarray}%
Motivated and inspired by the work of the above mentioned renowned
mathematicians, we will establish a new inequality by using weight function
, which will be better and generalized than those developed in $\left[ 1-4%
\right] .$ Some other interesting inequalities are also presented as special
cases. In the last, we presented applications for some special means and in
numerical integration.

\section{Main Results}

In order to prove our main result we first give the following essential
definition.

We assume that the weight function (or density) $w:(a,b)\longrightarrow
\lbrack 0,\infty )$ to be non-negative and integrable over its entire domain
and%
\begin{equation*}
\int_{a}^{b}w(t)dt<\infty .
\end{equation*}%
The domain of $\ w$\ \ \ may be finite or infinite and may vanish at the
boundary point. We denote the moment%
\begin{equation*}
m(a,b)=\int_{a}^{b}w(t)dt.
\end{equation*}%
We now give our main result.

\begin{theorem}
\textit{Let \ }$\mathit{\ }f:[a,b]\rightarrow
\mathbb{R}
$ \textit{be continuous on }$[a,b]$\textit{\ and differentiable on }$(a,b)$
and satisfy the condition $\theta \leq f^{\text{ }\prime }\leq \Phi $ $\ ,\
x\in (a,b).$ Then we have the inequality%
\begin{eqnarray}
&&\left\vert f(x)-\frac{1}{m(a,b)}w(x)\left( b-a\right) \left( x-\frac{a+b}{2%
}\right) f^{\text{ }\prime }(x)-\frac{1}{m(a,b)}\int_{a}^{b}f(t)w(t)dt\right%
\vert  \notag \\
&\leq &\frac{1}{2m^{2}(a,b)}w(x)\left( \frac{1}{2}\left( b-a\right)
^{2}+2\left( x-\frac{a+b}{2}\right) ^{2}\right)  \notag \\
&&  \notag \\
&&\times \left( \frac{1}{2}\left( b-a\right) +\left\vert x-\frac{a+b}{2}%
\right\vert \right) \left\Vert f^{\text{ }\prime \prime }\right\Vert _{w,1}
\label{2.1}
\end{eqnarray}%
for all $x\in \left[ a,b\right] $.
\end{theorem}

\begin{proof}
L\textit{et us define the mapping\ }$P(.,.):[a,b]\longrightarrow
\mathbb{R}
$ given by

\begin{equation*}
P(x,t)=\left\{
\begin{array}{lll}
\int_{a}^{t}w(u)du\text{ \ } & \text{{if}} & t\in \lbrack a,x] \\
\int_{b}^{t}w(u)du\text{ } & \text{{if}} & t\in (x,b].%
\end{array}%
\right.
\end{equation*}%
Integrating by parts, we have%
\begin{equation}
P(x,t)f^{\text{ }\prime }(t)dt=\int_{a}^{b}f(x)m(a,b)-\int_{a}^{b}f(t)w(t)dt.
\tag{2.2}
\end{equation}%
Applying the identity $(2.2)$ for $f^{\text{ }\prime }(.),$ we get%
\begin{equation*}
f^{\text{ }\prime }(t)=\frac{1}{m(a,b)}\int_{a}^{b}P(t,s)f^{\text{ }\prime
\prime }(s)ds+\frac{1}{m(a,b)}\int_{a}^{b}f^{\text{ }\prime }(s)w(s)ds.
\end{equation*}%
Substituting $f^{\text{ }\prime }(t)$ in the right membership of $\left(
2.2\right) ,$we have%
\begin{eqnarray}
f(x) &=&\frac{1}{m^{2}(a,b)}\int_{a}^{b}\int_{a}^{b}P(x,t)P(t,s)f^{\text{ }%
\prime \prime }(s)dsdt  \notag \\
&&+\frac{1}{m^{2}(a,b)}\int_{a}^{b}P(x,t)dt\int_{a}^{b}f^{\text{ }\prime
}(s)w(s)dsdt+\frac{1}{m(a,b)}\int_{a}^{b}f(t)w(t)dt.  \label{2.3}
\end{eqnarray}%
Since%
\begin{equation*}
\int_{a}^{b}P(x,t)dt=w(x)\left( b-a\right) \left( x-\frac{a+b}{2}\right)
\end{equation*}%
and%
\begin{equation*}
\int_{a}^{b}f^{\text{ }\prime }(s)w(s)ds=f^{\text{ }\prime }(x)m\left(
a,b\right) .
\end{equation*}%
From $\left( 2.3\right) $ therefore we obtain%
\begin{eqnarray}
f(x) &=&\frac{1}{m(a,b)}w(x)\left( b-a\right) \left( x-\frac{a+b}{2}\right)
f^{\text{ }\prime }(x)+\frac{1}{m(a,b)}\int_{a}^{b}f(t)w(t)dt  \notag \\
&&+\frac{1}{m^{2}(a,b)}\int_{a}^{b}\int_{a}^{b}P(x,t)P(t,s)f^{\text{ }\prime
\prime }(s)dsdt.  \label{2.4}
\end{eqnarray}%
Now%
\begin{equation*}
\int_{a}^{b}\text{ }\left\vert P(t,s)\right\vert ds=\frac{1}{2}w(t)\left[
\left( t-a\right) ^{2}+\left( t-b\right) ^{2}\right] ,
\end{equation*}%
\begin{eqnarray*}
&&\int_{a}^{b}\left\vert P(x,t)\right\vert \left[ \frac{w(t)}{2}\left(
\left( t-a\right) ^{2}+\left( b-t\right) ^{2}\right) \left\vert f^{\text{ }%
\prime \prime }(s)\right\vert ds\right] dt \\
&\leq &\frac{1}{2}w(x)\left( \left( x-a\right) ^{2}+\left( b-x\right)
^{2}\right) \max \left\{ t-a,b-t\right\} \left\Vert f^{\text{ }\prime \prime
}\right\Vert _{w,1}.
\end{eqnarray*}%
From $\left( 2.4\right) $, we have%
\begin{eqnarray}
&&\left\vert f(x)-\frac{1}{m(a,b)}w(x)\left( b-a\right) \left( x-\frac{a+b}{2%
}\right) f^{\text{ }\prime }(x)-\frac{1}{m(a,b)}\int_{a}^{b}f(t)w(t)dt\right%
\vert   \notag \\
&\leq &\frac{1}{2m^{2}(a,b)}w(x)\left( \left( x-a\right) ^{2}+\left(
b-x\right) ^{2}\right) \max \left\{ t-a,b-t\right\} \left\Vert f^{\text{ }%
\prime \prime }\right\Vert _{w,1}.  \label{2.5}
\end{eqnarray}%
Using%
\begin{equation*}
\max \left\{ t-a,b-t\right\} =\frac{1}{2}\left( b-a\right) +\left\vert x-%
\frac{a+b}{2}\right\vert
\end{equation*}%
in $\left( 2.5\right) $, we get our desired result.
\end{proof}

\begin{remark}
For $w\left( t\right) =1\mathbf{,}\ \ $the inequality $\left( 2.1\right) $
gives%
\begin{eqnarray}
&&\left\vert f(x)-\left( x-\frac{a+b}{2}\right) f^{\text{ }\prime }(x)-\frac{%
1}{\left( b-a\right) }\int\limits_{a}^{b}f(t)dt\right\vert  \notag \\
&\leq &\frac{1}{2\left( b-a\right) ^{2}}\left( \frac{1}{2}\left( b-a\right)
^{2}+2\left( x-\frac{a+b}{2}\right) ^{2}\right)  \notag \\
&&  \notag \\
&&\times \left( \frac{1}{2}\left( b-a\right) +\left\vert x-\frac{a+b}{2}%
\right\vert \right) \left\Vert f^{\text{ }\prime \prime }\right\Vert _{1}
\label{2.6}
\end{eqnarray}%
which is similar to \bigskip Barnett's result proved in \cite{1}.
\end{remark}

\begin{corollary}
Under the assumptions of Theorem $4$ and choosing $x=\frac{a+b}{2}$ , we
have the perturbed midpoint inequality%
\begin{eqnarray}
&&\left\vert f\left( \frac{a+b}{2}\right) -\frac{1}{m(a,b)}%
\int_{a}^{b}f(t)w(t)dt\right\vert  \notag \\
&\leq &\frac{1}{8m^{2}(a,b)}w(x)\left( b-a\right) ^{3}\left\Vert f^{\text{ }%
\prime \prime }\right\Vert _{w,1}.  \label{2.7}
\end{eqnarray}
\end{corollary}

\begin{proof}
This follows from inequality $\left( 2.1\right) .$
\end{proof}

\begin{corollary}
Under the assumptions of Theorem $4$, we have the perturbed trapezoidal
inequality%
\begin{eqnarray}
&&\left\vert \frac{f(a)+f(b)}{2}-\frac{1}{m(a,b)}\int_{a}^{b}f(t)w(t)dt+%
\frac{1}{m(a,b)}\frac{\left( b-a\right) ^{2}}{4}\left( w(a)f^{\text{ }\prime
}(a)-w(b)f^{\text{ }\prime }(b)\right) \right\vert  \notag \\
&\leq &\frac{1}{4m^{2}(a,b)}\left( b-a\right) ^{3}\left[ w(a)+w(b)\right]
\left\Vert f^{\text{ }\prime \prime }\right\Vert _{w,1}.  \label{2.8}
\end{eqnarray}
\end{corollary}

\begin{proof}
Put $x=a$ and $x=b$ in $\left( 2.1\right) $, summing up the obtained
inequalities, using the triangle inequality and dividing by 2, we get the
required inequality.
\end{proof}

\begin{remark}
The result given in $\left( 2.8\right) $ is different from the comparable
results available in \cite{4}.
\end{remark}

\begin{remark}
We can get the best estimation from the inequality $\left( 2.1\right) $ ,
only when $x=\frac{a+b}{2}$, this yields the inequality $\left( 2.7\right) .$
It shows that mid point estimation is better than the trapezoid estimation.\
\end{remark}

\section{\textbf{Application for some special means}}

We may now apply inequality $\left( 2.1\right) ,$ to deduce some
inequalities for special means by the use of particular mappings as follows:

\begin{remark}
Consider $f(x)=\sqrt{x}\ \ln x\ ,x\in \left[ a,b\right] \ \subset \left(
0,\infty \right) $

and%
\begin{equation*}
w(x)=\frac{1}{\sqrt{x}},
\end{equation*}%
The inequality $\left( 2.1\right) $ therefore gives%
\begin{eqnarray}
&&\left\vert
\begin{array}{c}
\sqrt{x}\ln x-\frac{1}{2\left( \sqrt{b}-\sqrt{a}\right) x}\left( b-a\right)
\left( x-A\right) \left( 1+\frac{1}{2}\ln x\right) \\
-\frac{1}{2\left( \sqrt{b}-\sqrt{a}\right) }\left( b-a\right) \ln I\left(
a,b\right)%
\end{array}%
\right\vert  \notag \\
&\leq &\frac{1}{8\left( \sqrt{b}-\sqrt{a}\right) ^{2}}\frac{1}{\sqrt{x}}%
\left( \frac{1}{2}\left( b-a\right) ^{2}+2\left( x-A\right) ^{2}\right)
\notag \\
&&\times \left( \frac{1}{2}\left( b-a\right) +\left\vert x-A\right\vert
\right) \frac{(b-a)}{4ab}\left( 1-\frac{\ln b^{a}-\ln a^{b}}{b-a}\right) .
\label{3.1}
\end{eqnarray}%
Choosing $x=A$ in $\left( 3.1\right) $, we get%
\begin{eqnarray}
&&\left\vert \sqrt{A}\ln A-\frac{1}{2\left( \sqrt{b}-\sqrt{a}\right) }\left(
b-a\right) \ln I\left( a,b\right) \right\vert  \notag \\
&\leq &\frac{1}{128ab\left( \sqrt{b}-\sqrt{a}\right) ^{2}}\frac{1}{\sqrt{x}}%
\left( b-a\right) ^{4}\left( 1-\frac{\ln b^{a}-\ln a^{b}}{b-a}\right) .
\label{3.2}
\end{eqnarray}%
\bigskip
\end{remark}

\begin{remark}
Consider \ $f(x)=\frac{1}{x}\sqrt{x}\ \ \ ,x\in \lbrack a,b]\subset \lbrack
1,\infty )$

and%
\begin{equation*}
w(x)=\frac{1}{\sqrt{x}}
\end{equation*}%
The inequality $\left( 2.1\right) $ therefore gives%
\begin{eqnarray}
&&\left\vert
\begin{array}{c}
\frac{1}{x}\text{ }\sqrt{x}\text{\ \ }+\frac{1}{4\left( \sqrt{b}-\sqrt{a}%
\right) }\frac{1}{x^{2}}\left( b-a\right) \left( x-A\right) \\
-\frac{1}{2\left( \sqrt{b}-\sqrt{a}\right) }\left( b-a\right) L_{-1}^{-1}%
\end{array}%
\right\vert  \notag \\
&\leq &\frac{1}{8\left( \sqrt{b}-\sqrt{a}\right) ^{2}}\frac{1}{\sqrt{x}}%
\left( \frac{1}{2}\left( b-a\right) ^{2}+2\left( x-A\right) ^{2}\right)
\notag \\
&&\times \frac{3}{8}\left( \frac{1}{2}\left( b-a\right) +\left\vert
x-A\right\vert \right) \left( \frac{b^{2}-a^{2}}{a^{2}b^{2}}\right) .
\label{3.3}
\end{eqnarray}%
Choosing $x=A$ in $\left( 3.3\right) $, we get%
\begin{eqnarray}
&&\left\vert \frac{1}{A}\text{ }\sqrt{A}\text{\ \ }-\frac{1}{2\left( \sqrt{b}%
-\sqrt{a}\right) }\left( b-a\right) L_{-1}^{-1}\right\vert  \notag \\
&\leq &\frac{3}{256\left( \sqrt{b}-\sqrt{a}\right) ^{2}}\frac{1}{\sqrt{A}}%
\left( b-a\right) ^{3}\left( \frac{b^{2}-a^{2}}{a^{2}b^{2}}\right) .
\label{3.4}
\end{eqnarray}
\end{remark}

\begin{remark}
Consider\ \ \ $f(x)=x^{p}\sqrt{x}$ $,$\ $x\in \left[ a,b\right] $ $f:\left(
0,\infty \right) \rightarrow R$\ ,$\ $where $p\in R/\left\{ -1,0\right\} $
then for $a<b,$

and%
\begin{equation*}
w(x)=\frac{1}{\sqrt{x}}
\end{equation*}%
The inequality $\left( 2.1\right) $ therefore gives%
\begin{eqnarray}
&&\left\vert
\begin{array}{c}
x^{p}\sqrt{x}\text{ }-\frac{1}{2\left( \sqrt{b}-\sqrt{a}\right) }\frac{1}{x}%
\left( b-a\right) \left( x-A\right) \left( p+\frac{1}{2}\right) x^{p} \\
-\frac{1}{2\left( \sqrt{b}-\sqrt{a}\right) }\left( b-a\right) L_{p}^{p}.%
\end{array}%
\right\vert  \notag \\
&\leq &\frac{1}{8\left( \sqrt{b}-\sqrt{a}\right) ^{2}}\frac{1}{\sqrt{x}}%
\left( \frac{1}{2}\left( b-a\right) ^{2}+2\left( x-A\right) ^{2}\right)
\notag \\
&&\times \left( \frac{1}{2}\left( b-a\right) +\left\vert x-A\right\vert
\right) \left( \frac{p^{2}-\frac{1}{4}}{p-1}\right) \left(
b^{p-1}-a^{p-1}\right) .  \label{3.5}
\end{eqnarray}%
Choosing $x=A$\ in $\left( 3.5\right) ,$ we get%
\begin{eqnarray}
&&\left\vert A^{p}\sqrt{A}\text{ }-\frac{1}{2\left( \sqrt{b}-\sqrt{a}\right)
}\left( b-a\right) L_{p}^{p}\right\vert  \notag \\
&\leq &\frac{1}{32\left( \sqrt{b}-\sqrt{a}\right) ^{2}}\frac{1}{\sqrt{A}}%
\left( b-a\right) ^{3}\left( \frac{p^{2}-\frac{1}{4}}{p-1}\right) \left(
b^{p-1}-a^{p-1}\right) .  \label{3.6}
\end{eqnarray}
\end{remark}

\section{\textbf{An Application to Numerical integration}}

Let $I_{n}:$ $a=x_{0}<x_{1}<x_{2}<....<x_{n-1}<x_{n}=b$ be a division of the
interval $[a,b]$ and $\ \ \xi =\left( \xi _{0},\xi _{1},......,\xi
_{n-1}\right) ,$\ a sequence of intermediate points \ $\xi _{i}\in \left[
x_{i},x_{i+1}\right] $ $\left( i=0,1,.....,n-1\right) .$ Consider the
perturbed Riemann sum defined by%
\begin{equation}
A=\underset{i=0}{\overset{n-1}{\sum }}m\left( x_{i},x_{i+1}\right) f\left(
\xi _{i}\right) -\underset{i=0}{\overset{n-1}{\sum }}w(\xi _{i})h_{i}\left(
\xi _{i}-\frac{x_{i}+x_{i+1}}{2}\right) f^{%
{\acute{}}%
}\left( \xi _{i}\right)  \tag{4.1}
\end{equation}

\begin{theorem}
\textit{Let \ }$\mathit{\ }f:[a,b]\rightarrow
\mathbb{R}
$ \textit{be continuous on }$[a,b]$\textit{\ and differentiable on }$(a,b)$,
such that $f^{\text{ }%
{\acute{}}%
}:\left( a,b\right) \rightarrow
\mathbb{R}
$ is bounded on $\left( a,b\right) $ and assume that\ $\gamma \leq $\ \ $f$ $%
^{%
{\acute{}}%
}\leq \Gamma $ for all \ $x\in \left( a,b\right) .$ $f^{\prime \prime
}:(a,b)\longrightarrow
\mathbb{R}
$ belongs \ to $\mathbf{L}_{1}(a,b),$\ \ i.e.%
\begin{equation*}
\left\Vert f^{\prime \prime }\right\Vert
_{w,1}:=\int\limits_{a}^{b}\left\vert w(t)f(t)\right\vert dt<\infty .
\end{equation*}%
we have%
\begin{equation}
\int\limits_{a}^{b}f(t)w(t)dt=A\left( f,I,w,\xi \right) +R\left( f,I,w,\xi
\right) ,  \tag{4.2}
\end{equation}%
where the remainder $R$ satisfies the estimation%
\begin{eqnarray}
&&\left\vert R\left( f,I,w,\xi \right) \right\vert  \notag \\
&\leq &\frac{\left\Vert f^{\text{ }\prime \prime }\right\Vert _{w,1}}{%
2m\left( x_{i},x_{i+1}\right) }w(\xi _{i})\left( \frac{1}{2}\left(
h_{i}\right) ^{2}+2\left( \xi _{i}-\frac{x_{i}+x_{i+1}}{2}\right) ^{2}\right)
\notag \\
&&\times \left( \frac{1}{2}\left( h_{i}\right) +\left\vert \xi _{i}-\frac{%
x_{i}+x_{i+1}}{2}\right\vert \right)  \label{4.3}
\end{eqnarray}%
for any choice $\xi $ of the intermediate points.
\end{theorem}

\begin{proof}
Apply Theorem $4$ on the interval $[x_{i},x_{i+1}]$,\ \ $\xi _{i}\in \lbrack
x_{i},x_{i+1}],$ \ where $h_{i}=x_{i+1}-x_{i}$ $\left( i=1,2,3....n-1\right)
,$ to get%
\begin{eqnarray*}
&&\left\vert m(x_{i},x_{i+1})\text{ }f(\xi _{i})-\overset{x_{i+1}}{\underset{%
x_{i}}{\int }}f(t)w(t)dt-(\xi _{i}-\frac{x_{i}+x_{i+1}}{2})w(\xi _{i})\left(
x_{i+1}-x_{i}\right) f^{%
{\acute{}}%
}\left( \xi _{i}\right) \right\vert \\
&\leq &\frac{\left\Vert f^{\text{ }\prime \prime }\right\Vert _{w,1}}{%
2m\left( x_{i},x_{i+1}\right) }w(\xi _{i})\left( \frac{1}{2}\left(
x_{i+1}-x_{i}\right) ^{2}+2\left( \xi _{i}-\frac{x_{i}+x_{i+1}}{2}\right)
^{2}\right) \\
&&\times \left( \frac{1}{2}\left( x_{i+1}-x_{i}\right) +\left\vert \xi _{i}-%
\frac{x_{i}+x_{i+1}}{2}\right\vert \right)
\end{eqnarray*}

for all $\xi _{i}\in \lbrack x_{i},x_{i+1}]$ and $i\in \left(
0,1,....n-1\right) .$ Summing the above two inequalities over $i$ from $0$
to $n-1$ and using the generalized triangular inequality, we get the desired
estimation.
\end{proof}

\section{Conclusions}

We established weighted Ostrowski type inequality for bounded differentiable
mappings which generalizes the previous inequalities developed and discussed
in \cite{1},\cite{3},\cite{5} and \cite{8}. Perturbed midpoint and trapezoid
inequalities are obtained. Some closely new results are also given. This
inequality is extended to account for applications in some special means and
numerical integration to show his applicability towards obtaining direct
relationship of these means. These generalized inequalities will also be
useful for the researchers working in the field of the approximation theory,
applied mathematics, probability theory, stochastic and numerical analysis
to solve their problems in engineering and in practical life.

\end{document}